\documentclass{elsarticle}
\usepackage{amssymb,amsmath,amsthm,latexsym,xcolor}
\usepackage{lscape}
\usepackage{hyperref}
\hypersetup{colorlinks=true, linktocpage,citebordercolor=blue,linkbordercolor=blue}
\bibpunct{\color{blue}[}{\color{blue}]}{,}{n}{}{;}
\newcommand{\bb[1]}{\textcolor{blue}{[#1]}}
\def\eqref#1{\textcolor{blue}{(\ref{#1})}}
\journal{Journal of Number Theory}
\newcommand\Perms[2]{\tensor[^{#1}]P{_{#2}}}
\newtheorem{theorem}{Theorem}[section]
\newtheorem{corollary}[theorem]{Corollary}
\newtheorem{definition}[theorem]{Definition}
\newtheorem{lemma}[theorem]{Lemma}
\newtheorem{proposition}[theorem]{Proposition}

\usepackage{tensor}

\begin{document}

\begin{frontmatter}

\title{Shifted Euler constants and a generalization of Euler-Stieltjes constants}

\author{Tapas Chatterjee\fnref{myfootnote}}
\corref{mycorrespondingauthor}
\cortext[mycorrespondingauthor]{Corresponding author}
\fntext[myfootnote]{Research of the first author was supported by a NBHM Research Project Grant-in-Aid with grant no. NBHM/2/11/39/2017/R$\&$D II/3481.}
\ead{tapasc@iitrpr.ac.in}
\author{Suraj Singh Khurana\fnref{myfootnote2}}
\address{Department of Mathematics, Indian Institute of Technology Ropar, Punjab-14001, India}
\fntext[myfootnote2]{Research of the second author was supported by Council of Scientific and Industrial Research (CSIR), India under File No: 09/1005(0016)/2016-EMR-1.}
\ead{suraj.khurana@iitrpr.ac.in}

\begin{abstract}
The purpose of this article is twofold. First, we introduce the  constants $\zeta_k(\alpha,r,q)$ where $\alpha \in (0,1)$ and study them along the lines of work done on Euler constant in arithmetic progression $\gamma(r,q)$ by  Briggs, Dilcher, Knopfmacher, Lehmer and some other authors. These constants are used for evaluation of certain integrals involving error term for Dirichlet divisor problem with congruence conditions and also to provide a closed form expression for the value of a class of Dirichlet L-series at any real critical point. In the second half of this paper, we consider the behaviour of the Laurent Stieltjes constants $\gamma_k(\chi)$ for a principal character $\chi.$ In particular we study a generalization of the ``Generalized Euler constants'' introduced by Diamond and Ford in 2008. We conclude with a short proof for a closed form expression for the first generalized Stieltjes constant $\gamma_1(r/q)$ which was  given by Blagouchine in 2015.
\end{abstract}

\end{frontmatter}

\section{Introduction}
It is well known that the Euler's constant $\gamma$ occurs as the constant term in the Laurent series expansion of the Riemann zeta function $\zeta(s)$ at $s=1$. In particular we have
\begin{equation} \label{LSC}
\zeta(s)=\dfrac{1}{s-1}+\sum\limits_{k=0}^{\infty}\dfrac{(-1)^{k}}{k!}\gamma_{k}(s-1)^k
\end{equation} 
where $\gamma_0$ is the Euler's constant $\gamma$ and in general for $k \geq 0$ the constant $\gamma_k$ is known as Euler-Stieltjes constant which is given by the limit
\begin{equation} \label{SRG}
\gamma_{k} := \lim_{N \rightarrow \infty}\left( \sum\limits_{n=1}^{N}\dfrac{\log^k n}{n}-\dfrac{\log^{k+1} N}{k+1}\right).
\end{equation} 
Not only this, the constant $\gamma$ has made its appearance in numerous other works \cite{Lag}  and is therefore considered as fundamental as $\pi$ and $e.$ But  unlike $\pi$ and $e$, the question of irrationality of $\gamma$ is still open. Consequently, many authors have  considered various generalizations of $\gamma$ and studied their properties leading to a vast literature. In 
\hyperref[SECS]{Section \ref{SECS}}, we briefly review a few generalizations of $\gamma$ important for our discussion  and then state the definition of the ``Shifted Euler constants" $\zeta_k(\alpha,r,q)$ along with the related results. The proofs of these results and some properties of $\zeta_k(\alpha,r,q)$ are given in the \hyperref[POTR]{Section \ref{POTR}}. In \hyperref[GOEC]{Section \ref{GOEC}}, we consider a generalization of ``Generalized Euler constants" introduced by Diamond and Ford in \cite{DHF} in the context of Dirichlet L-series. At the end we give a closed form expression for the first Generalized Stieltjes constant which occurs in the Laurent series expansion of Hurwitz zeta function about the point $s=1.$
\subsection{Notations}
To facilitate our discussion we provide list of the abbreviated notations which will be used throughout this paper. Empty set will be denoted by the symbol $\emptyset$. The value of an empty sum and an empty product will be considered as $0$ and $1$ respectively. The symbol $\mathbb{N}_0$ stands for the set of all non-negative integers. For a complex number $z$, let $\Re(z)$ denote the real part of $z$. The symbol $\Perms{n}{i}$ will denote the value $\dfrac{n!}{(n-i)!}$. The residue of the function $f(z)$ at the point $z=a$ will be written as $\text{Res}_{z=a}f(z)$. The symbol $D_{\alpha}$ will be used to denote the region $\{s \in \mathbb{C} \mid  |s - \alpha | < | \alpha-1 |  \}$. Notations $\zeta(s)$, $\psi(s)$ and $\zeta(s,x)$ denote respectively the Riemann zeta function, the Digamma function and the Hurwitz zeta function.  $F(x,s)$ represents the periodic zeta function and is defined by the series
\begin{equation*}  
 F \left(x,s\right)=\sum_{n=1}^{\infty}\frac{e^{2\pi inx}}{n^{s}}
\end{equation*} where $\Re(s)>1$ if $x \in \mathbb{Z}$ and $\Re(s)>0$ otherwise. For $r,q \in \mathbb{N}$ where $r \leq q$ the partial zeta function $\zeta(s;r,q)$ is defined as
\begin{equation*}
\zeta(s;r,q):=\sum_{\substack{n \geq 1 \\ n \equiv r \bmod q}}\dfrac{1}{n^s}
\end{equation*}
for $\Re(s)>1.$ For a function $f$, the notation $f^{(n)}$ will mean the $n$-th derivative of the function $f$. In particular, $f^{(0)}=f$. For any real number $x$, $\{x\}$ will represent the fractional part of $x$ and $\lfloor x \rfloor$ will denote the greatest integer less than or equal to $x$. Unless otherwise stated the symbol $p_i$ will denote the $i$-th prime. The symbol $\mathcal{P}_r$ denotes the set consisting of first $r$ many primes. $\bar{\mathbb{Q}}$ represents the field  consisting of all algebraic numbers. Let $\alpha \in (0,1)$, $k$ be a non negative integer, $r,q \in \mathbb{N}$ where $r \leq q$. Then the Euler-Stieltjes constant $\gamma_k$, the Euler-Lehmer constant $\gamma(r,q)$, the Generalized Euler-Lehmer constant $\gamma_k(r,q)$ and the shifted Euler-Lehmer constant $\zeta_k(\alpha,r,q)$ are given by \eqref{SRG}, \eqref{ELC}, \eqref{GELC} and Definition \ref{MND} respectively.  
\section{Shifted Euler constants} \label{SECS}
In 1961, Briggs \cite{BRI} considered the constants $\gamma(r,q)$ associated with arithmetic progressions defined as
\begin{equation} \label{ELC}
\gamma(r,q):=\lim_{x\rightarrow\infty}\left(\sum\limits_{\substack{0<n\leq x \\ n \equiv r\bmod q}} \dfrac{1}{n}-\frac{1}{q}\log x\right)
\end{equation}
where $1 \leq r \leq q.$
It is easy to see that $\gamma(1,1)=\gamma.$ In \cite{Le}, using discrete Fourier transforms and some basic tools Lehmer obtained many properties of the constants $\gamma(r,q)$  and derived  an elementary proof of the well known Gauss theorem on digamma function $\psi(z)$ at rational arguments. Further the connection of $\gamma(r,q)$ with the class numbers of quadratic fields $\mathbb{Q}(\sqrt{ \pm q})$ and certain infinite series was given in \cite{Le}. In particular \hyperlink{cite.Le}{\cite[Theorem 8]{Le}}, it was shown that for a $q$ periodic arithmetic function $f$ satisfying $\sum\limits_{n=1}^{q}f(n)=0$ the Dirichlet series $L(s,f)$ converges at $s=1$ and is given by the following closed form expression
\begin{equation*}
L(1,f)=\sum\limits_{r=1}^{q}f(r)\gamma(r,q).
\end{equation*}
The constants $\gamma(r,q)$ are referred to as Euler-Briggs-Lehmer constants or 
sometimes just Euler-Lehmer constants and results related to their arithmetic nature has been given by Murty and Saradha in \cite{EM}. For more results related to arithmetic nature of $\gamma(r,q)$ and its generalization see \bb[\citenum{GE},\citenum{GSS}]. For the work done on  p-adic version of the Euler-Lehmer constants $\gamma(r,q)$ see \cite{CTG} and \cite{DIJ}. 

A further generalization of $\gamma(r,q)$ was introduced by Knopfmacher in \cite{KNO} and later studied in detail by Dilcher in \cite{Di}. They considered the generalized Euler-Lehmer constants of higher order which are defined as 
\begin{equation} \label{GELC}
\gamma_k(r,q):=\lim_{x\rightarrow \infty}\left(\sum_{\substack{n\leq x \\ n \equiv r \bmod q}}\dfrac{\log^k n}{n}-\dfrac{\log^{k+1} x}{q(k+1)}\right).
\end{equation}
It is easy to see that $\gamma_0(1,1)=\gamma$ and $\gamma_0(r,q)=\gamma(r,q)$ and $\gamma_k(1,1)=\gamma_k.$ Most of the results given in \cite{Le} were generalized in \cite{KNO} using the properties of $\gamma_k(r,q).$ In Proposition 9 of \cite{KNO} it was shown that for a $q$ periodic arithmetic function $k$-th derivative of $L(s,f)$  at $s=1$ exists if and only if $\displaystyle\sum\limits_{r=1}^{q}f(r)=0$. Further in the case of existence, the value is given by the following closed form expression
 \begin{equation*}
 L^{(k)}(1,f)=(-1)^{k}\sum\limits_{r=1}^{q}f(r)\gamma_{k}(r,q).
 \end{equation*}
For some particular cases where the above identity was used to give explicit expressions  the reader may see  section 6 of \cite{KNO}.   
 
From the above discussion it seems natural to ask for similar results related to the arithmetic nature or closed form expressions for $L(s,f)$  at points other than $s=1$ as well. We investigate this question for real points lying in the critical strip $0<\Re(s)<1$ which is an important region to study for many L-functions. For this we consider a variant of generalized Euler-Lehmer constants and study its properties along the lines of contributions made by Briggs, Dilcher, Knopfmacher, Lehmer and some other authors.   
\begin{definition} \label{MND}
For $\alpha \in (0,1)$, $k \in \mathbb{N}_0$ and $r,q \in \mathbb{N}$ where $r \leq q$, the Shifted Euler constant $\zeta_k(\alpha,r,q)$ is defined as the limit
\begin{equation*} 
\zeta_k(\alpha,r,q):= \lim_{x \rightarrow \infty}\left( H_{k}(x,\alpha,r,q) - \dfrac{I_k(x,\alpha)}{q} \right)
\end{equation*}
where 
$$I_k(x, \alpha)=\sum\limits_{i=0}^{k}\left((-1)^i\dfrac{\Perms{k}{i}\log^{k-i}x}{(1-\alpha)^{i+1}}\right)x^{1-\alpha} $$
and
$$H_{k}(x,\alpha,r,q):= \sum\limits_{\substack{n\leq x \\ n \equiv r\bmod q}} \dfrac{\log^{k} n}{n^{\alpha}}.$$
Here $\Perms{n}{i}:=\dfrac{n!}{(n-i)!}$.
\end{definition}
For the existence of the limit $\zeta_{k}(\alpha,r,q)$ see \hyperref[MLET]{Proposition \ref{MLET}}. It is easy to see from the definition that $\zeta_0(\alpha,1,1)$ is equal to the value $\zeta(\alpha)$\hyperlink{cite.Ap}{\cite[p. 56]{Ap}}. More generally the constants $\zeta_k(\alpha,r,q)$ are related with the periodic zeta function $F(x,s)$ which is defined by the series
\begin{equation*}  
 F \left(x,s\right)=\sum_{n=1}^{\infty}\frac{e^{2\pi inx}}{n^{s}}
\end{equation*} where $\Re(s)>1$ if $x \in \mathbb{Z}$ and $\Re(s)>0$ otherwise. This function has an analytic continuation throughout whole complex plane except when $x \in \mathbb{Z}$ in which case it has a simple pole at $s=1$. The following theorem gives an expression for $\zeta_k(\alpha,r,q)$ in terms of derivatives of $F(x,s)$.
 \begin{theorem} \label{MST}
For $\alpha \in (0,1)$, $k \in \mathbb{N}_0$ and $b,q \in \mathbb{N}$ such that $b \leq q$ we have the following 
\begin{equation} \label{DPZ}
\zeta_k(\alpha,b,q)=\dfrac{H_k(\alpha)}{q}+\dfrac{(-1)^k}{q}\sum\limits_{a=1}^{q-1}F^{(k)} \left(\frac{a}{q},\alpha \right)e^{\frac{-2\pi iab}{q}}
\end{equation}
where 
\begin{equation*}
H_k(\alpha)= \int\limits_{1}^{\infty} \{t\}\left( \dfrac{k \log^{k-1} t-\alpha  \log^{k} t}{t^{1+\alpha}}  \right) dt 
+ \dfrac{(-1)^{k+1}k!}{(1-\alpha)^{k+1}}+ \left\lfloor \dfrac{1}{k+1} \right\rfloor
\end{equation*} and $F\left(x,s\right)$ is the periodic zeta function.
\end{theorem}
In the special case when $k=0$, using analytic continuation of $\zeta(s)$ we obtain a series representation for $\zeta_0(\alpha,b,q).$
\begin{corollary} \label{MSC}
 For $\alpha \neq 1$ and $\Re (\alpha) >0$ we have the following 
 \begin{equation*}
 \zeta_0(\alpha,b,q)=\dfrac{\alpha}{q(\alpha-1)}+\dfrac{1}{q(1-\alpha)}\sum\limits_{t \geq 1}  \sum\limits_{n=2}^{\infty} \dfrac{(\alpha-1)_t}{n^{\alpha+t}} +\dfrac{1}{q}\sum\limits_{a=1}^{q-1}\sum\limits_{n=1}^{\infty}\dfrac{e^{\frac{2\pi i a(n-b)}{q}}}{n^{\alpha}}
 \end{equation*}
 where $(s)_t=\dfrac{s(s+1)\cdots(s+t)}{(t+1)!}$ is the classical Pochhammer symbol.
\end{corollary}
To motivate for the next result we recall that the Laurent series coefficients  $\gamma_k$ of the function $\zeta(s)$ around the point $s=1$ is well known and has been studied extensively by many authors \bb[\citenum{ADT},\citenum{WB},\citenum{BC},\citenum{MWC},\citenum{MC},\citenum{MCK},\citenum{LJT},\citenum{ZNW}]. On the contrary, the Laurent series expansion of certain Dirichlet L-functions at points other than the poles seems to be first considered by A. Ivi\'c \cite{ALI} and was used to evaluate integrals containing error terms related to some well known problems \hyperlink{cite.ALI}{\cite[Section 4]{ALI}}. In \cite{DMC}, using the Maclaurin series expansion of $\zeta(s)$ at $s=0$, Lehmer gave an expression for the infinite sum $\sum\limits_{\rho}\frac{1}{\rho^k}$ where $\rho$ varies over all the complex zeroes of $\zeta(s).$ For Maclaurin series expansion of Hurwitz zeta function at $s=0$ see \cite{BKH}. Here we give explicit expressions for the Laurent series expansion of some well known Dirichlet L-functions at points lying on the real line in the critical strip. For this, we make use of the partial zeta function $\zeta(s;r,q)$ which is defined as 
\begin{equation*}
\zeta(s;r,q):=\sum_{\substack{n \geq 1 \\ n \equiv r \bmod q}}\dfrac{1}{n^s}
\end{equation*}
for $\Re(s)>1.$  Using the well known meromorphic continuation of Hurwitz zeta function $\zeta(s,\alpha)$ \hyperlink{cite.Ap}{\cite[Chapter 12]{Ap}} and the identity 
\begin{equation} \label{HTS}
\zeta(s;r,q)=q^{-s}\zeta\left(s,\dfrac{r}{q}\right)
\end{equation}
one can easily deduce the meromorphic continuation of $\zeta(s;r,q)$ over the whole complex plane except at the point $s=1.$ Recently Shirasaka gave \bb[\citenum{SHI}, Theorem(i)] the Laurent series expansion of the function $\zeta(s;r,q)$ around the point $s=1$ as follows:
\begin{equation} \label{LSH}
\zeta(s;r,q)=\dfrac{1}{q(s-1)}+\sum\limits_{k=0}^{\infty}\dfrac{(-1)^{k}}{k!}\gamma_k(r,q)(s-1)^k.
\end{equation}
It is easy to see that the expansion given by the equation \eqref{LSH} is a generalization of the expansion given in  \eqref{LSC}. Using  \eqref{LSH} and some properties of Hurwitz zeta function Shirasaka derived identities of Lehmer \cite{Le}, Dilcher \cite{Di} and Kanemitsu \cite{KAN} in a unified manner. 
We recall that for $\alpha \in (0,1)$ the notation $D_{\alpha}$ represents the region 
\begin{equation*}
D_\alpha:=\{s \in \mathbb{C} \mid |s - \alpha | < | \alpha-1 |  \}.
\end{equation*}
\begin{theorem} \label{TSF} 
For  $\alpha \in (0,1)$ and $r,q \in \mathbb{N}$ the Taylor series expansion of the function $\zeta(s;r,q)$ at the point $s=\alpha$ is given by 
\begin{equation} \label{TSP}
\zeta(s;r,q)=\sum\limits_{m=0}^{\infty}\dfrac{(-1)^m}{m!}\zeta_m(\alpha,r,q)(s-\alpha)^m
\end{equation}
and is valid for $s \in D_\alpha.$ 
\end{theorem}
Some immediate corollaries of the above theorem are the following.
\begin{corollary}
Let $k \geq 0$ and $d$ be a common divisor of $r$ and $q$. Then we have 
\begin{equation*} 
\zeta_k(\alpha,r,q)=\sum\limits_{j=0}^{k}\dbinom{k}{j}\dfrac{\log^{k-j} d}{d^{\alpha}}  \zeta_j\left(\alpha,\frac{r}{d},\frac{q}{d}\right).
\end{equation*}
\end{corollary}
\begin{proof}
Follows from the equality of the Laurent series expansion  of the two functions $\zeta(s;r,q)$ and $\frac{1}{d^s}\zeta(s;\frac{r}{d},\frac{q}{d})$ at the point $s=\alpha$.
\end{proof}
\begin{corollary}
The Taylor series expansion of Hurwitz zeta function $\zeta\left(s,\dfrac{r}{q}\right)$ around the point $s=\alpha$ is given by 
\begin{equation*}
\zeta\left(s,\dfrac{r}{q}\right)=\sum\limits_{m=0}^{\infty}\left( q^{\alpha}\sum\limits_{j=0}^{m}(-1)^{m-j}\dfrac{\log^j q}{j!}\dfrac{\zeta_{m-j}(\alpha,r,q)}{(m-j)!} \right)(s-\alpha)^m
\end{equation*}
and is valid for $s \in D_\alpha.$ 
\end{corollary}
\begin{proof}
Expand $q^{-s}$ at the point $s=\alpha$ in the identity \eqref{HTS} and then use the above theorem.
\end{proof}
In the special  case when $r=1$ and $q=1$ we recover from Eq.~\eqref{TSP}, the expansion of $\zeta(s)$ as given in the equation 2.3 of \cite{ALI}\footnote{There seems to be a missing term $\frac{k!}{(s^{'}-1)^{k+1}}$ in the expression 2.4 for $\gamma_k(s^{'})$}. Hence, the constants $\zeta_k(\alpha,1,1)$ coincides with $\gamma_k(s^{'})$\hyperlink{cite.ALI}{\cite[Equation 2.4]{ALI}} which were used to express integrals involving error terms of Piltz divisor problem \hyperlink{cite.ALI}{\cite[Section 4.1]{ALI}}.
For work done related to Dirichlet and Piltz divisor problem see \hyperlink{cite.ARZ}{\cite[Chapter 13]{ARZ}} and \hyperlink{cite.TITC}{\cite[Chapter 12]{TITC}}. Here we consider a more general Dirichlet divisor problem with congruence conditions \bb[\citenum{LZM},\citenum{LKUI},\citenum{MWN},\citenum{NWGT},\citenum{NWG}] which is the study of the error term $\Delta_{2}(x;r_1,q_1,r_2,q_2)$ defined by the equation   
\begin{equation} \label{ETE}
\Delta_{2}(x;r_1,q_1,r_2,q_2)= \sum\limits_{n \leq x}d_{*}(n)-\text{Res}_{s=1}\left( \zeta(s;r_1,q_1)\zeta(s;r_2,q_2)\dfrac{x^s}{s} \right)
\end{equation}
where $d_{*}(n)$ is the number of elements in the set
\begin{align*}
\{ (n_1,n_2) \in \mathbb{N}^2 : n=n_1n_2, n_1 \equiv r_1 \bmod q_1, n_2 \equiv r_2 \bmod q_2 \}.
\end{align*}
It is known from Richert's work \cite{RIC} and Huxley's estimates \cite{HUX} that 
\begin{equation} \label{BFDE}
\Delta_2(q_1q_2x;r_1,q_1,r_2,q_2) \ll x^{\frac{131}{416}}(\log x)^{\frac{26947}{8320}}.
\end{equation}
Here we express a family of integrals involving the error term $\Delta_2(x;r_1,q_1,r_2,q_2)$ similarly to the results obtained for $\Delta_k$ \bb[\citenum{FTA},\citenum{LAV},\citenum{SIT}] in the case of classical Dirichlet divisor problem.  
\begin{theorem} \label{ETT}
Let $\tilde{\Delta}_2(x)$ denote the error term $\Delta_2(x;r_1,q_1,r_2,q_2)$ given by \eqref{ETE}. Then for any non negative integer $k$ we have
\begin{equation} \label{IEE}
\begin{split}
\int_{1}^{\infty}&\dfrac{\log^kx - k \log^{k-1}x}{x^2}\tilde{\Delta}_{2}(x) dx  \\
&= \sum\limits_{i=0}^{k}\dbinom{k}{i}\gamma_{k-i}(r_2,q_2)\gamma_i(r_1,q_1)  + \left\lfloor \dfrac{1}{k+1} \right\rfloor \left( \dfrac{1}{q_1q_2}-\dfrac{\gamma_0(r_1,q_1)}{q_2} - \dfrac{\gamma_0(r_2,q_2)}{q_1} \right) \\
& \quad - \dfrac{\gamma_{k+1}(r_2,q_2)}{q_1(k+1)}-\dfrac{\gamma_{k+1}(r_1,q_1)}{q_2(k+1)}
\end{split}
\end{equation}
and for $ \alpha \in \left(\dfrac{131}{416},1\right)$ we have
\begin{equation} \label{IEE2}
\begin{split}
\int_{1}^{\infty}&\dfrac{\alpha\log^kx - k \log^{k-1}x}{x^{1+\alpha}}\tilde{\Delta}_{2}(x) dx \\
&= \sum\limits_{l=0}^{k}\zeta_l(\alpha,r_1,q_1)\zeta_{k-l}(\alpha,r_2,q_2) +\left\lfloor \dfrac{1}{k+1} \right\rfloor \left( \dfrac{1}{q_1q_2}-\dfrac{\gamma_0(r_1,q_1)}{q_2} - \dfrac{\gamma_0(r_2,q_2)}{q_1} \right) \\
& - \dfrac{1}{q_1q_2}\left(\dfrac{(k+1)!}{(\alpha-1)^{k+2}}\right) - \left( \dfrac{\gamma_0(r_1,q_1)}{q_2} + \dfrac{\gamma_0(r_2,q_2)}{q_1} \right)\left(\dfrac{k!}{(\alpha-1)^{k+1}}\right).
\end{split}
\end{equation} 
\end{theorem}
In the special case when $k=0$ and $r_1=q_1=r_2=q_2=1$ the identity \eqref{IEE} reduces to the
one given by Lavrik, Israilov and \"Edgorov in \cite{LAV}\footnote{The authors in \cite{LAV} use a slightly different definition for the constant $\gamma_k$.}. We remark here that by using the idea of Sitaramachandra Rao in \cite{SIT}, the proof of \hyperref[ETT]{Theorem \ref{ETT}} can be generalized for the general error term $\tilde{\Delta}_k$.  \\

 Now we discuss the next consequence of \hyperref[TSF]{Theorem \ref{TSF}} in connection with the Dirichlet series with periodic coefficients.For a $q$ periodic arithmetic function $f$, the Laurent series expansion of $L(s,f)$ at the   
point $s=1$ was given by Ishibashi and Kanemitsu in \cite{IKA}. They showed that \hyperlink{cite.IKA}{\cite[Theorem 2]{IKA}} the Laurent(or Taylor) series expansion of $L(s,f)$ at $s=1$ is given by 
\begin{equation} \label{LSEFP}
L(s,f)=\dfrac{\gamma_{-1}(f)}{s-1}+\sum\limits_{n=1}^{\infty}\dfrac{\gamma_{n}(f)}{n!}(s-1)^{n}
\end{equation}
where a closed form expression of $\gamma_{i}$ is explicitly given for $i=-1,0,1,2$. 
Such an expansion was used to study problems related to the product of L functions \hyperlink{cite.IKA}{\cite[Section 2]{IKA}}. We state the following result regarding the Laurent series expansion for $L(s,f)$ which follows directly from the identity 
\begin{equation*}
L(s,f)=\sum\limits_{r=1}^{q}f(r)\zeta(s;r,q)
\end{equation*}
and \hyperref[TSF]{Theorem \ref{TSF}}.
\begin{theorem} \label{PLE}
For $\alpha \in (0,1)$ and a $q$ periodic arithmetic function $f$, the Dirichlet L function $L(s,f)$ has the following Taylor series expansion at the point $s=\alpha$  
\begin{equation} \label{PLEE}
L(s,f)=\sum\limits_{k=0}^{\infty}\dfrac{(-1)^k}{k!}\left( \sum\limits_{r=1}^{q} f(r)\zeta_k(\alpha,r,q) \right)(s-\alpha)^k
\end{equation}
which is valid for $s \in D_\alpha$. Moreover, the $k$-th derivative of $L(s,f)$ at $s=\alpha$ has the following closed form expression
\begin{equation} \label{LKN}
L^{(k)}(\alpha,f)=(-1)^k \sum\limits_{r=1}^{q} f(r)\zeta_k(\alpha,r,q).
\end{equation}
\end{theorem}
An infinite series expression for the constants $\zeta_k(\alpha,r,q)$ can be given as follows.
\begin{corollary} \label{PLEC}
For any $a$ such that $(a,q)=1$ and $1\leq a \leq q$  we have
\begin{equation*}
\zeta_m(\alpha,a,q)=\dfrac{m!}{q(\alpha-1)^{m+1}}+\dfrac{1}{\phi(q)}\sum\limits_{k \geq m}\sum\limits_{\chi}\dfrac{(-1)^{k}\bar{\chi}(a)\gamma_k(\chi)}{(k-m)!}(1-\alpha)^{k-m}
\end{equation*}
where $\chi$ varies over all the Dirichlet characters modulo $q$.
\end{corollary}
\begin{proof}
Using \eqref{LSEFP} and \eqref{PLEE} we can obtain two expressions for $L^{(m)}(\alpha,\chi).$ Since they must be equal we get 
\begin{equation*}
\sum\limits_{r=1}^{q}\chi(r)\zeta_m(\alpha,r,q)=\dfrac{\gamma_{-1}(\chi)m!}{(\alpha-1)^{m+1}}+\sum\limits_{k \geq m} \dfrac{(-1)^k\gamma_k(\chi)}{(k-m)!}(1-\alpha)^{k-m}.
\end{equation*}
Now use the well known orthogonality relation of Dirichlet characters $\chi$ to deduce the result.
\end{proof}
The value of $L(s,\chi)$ and its derivatives at the point $s=\frac{1}{2}$ is of special interest to many authors \bb[\citenum{BMV},\citenum{BMI},\citenum{ISA},\citenum{MVA},\citenum{SOK}]. It is conjectured that $L(\frac{1}{2},\chi) \neq 0$ for all primitive Dirichlet characters $\chi.$ Recently Murty and Tanabe \cite{MTA} studied the relation between the arithmetic nature of $e^{\gamma}$ and non vanishing of central values of Artin L functions. From \hyperref[PLE]{Theorem \ref{PLE}} we can observe the following condition.
\begin{corollary}
For any $k$ if the set $\{\zeta_k(\frac{1}{2},r,q) \mid 1\leq r \leq q \}$ is linearly independent over $\bar{\mathbb{Q}}$ then $L^{(k)}(\frac{1}{2},\chi) \neq 0$ for all Dirichlet character $\chi$ modulo $q$.
\end{corollary} 
\begin{proof}
From the hypothesis it is clear that $0$ does not belong to the given set. Now suppose on the contrary that $L^{(k)}(\frac{1}{2},\chi)=0$. Then from \eqref{LKN} it follows that when $\alpha=1/2$ and $f=\chi$, there exists a non trivial linear combination from the set $\{\zeta_k(\frac{1}{2},r,q) \mid 1\leq r \leq q \}$ over $\bar{\mathbb{Q}}$ which is equal to 0. This contradicts the linear independence of the given set. 
\end{proof}
\section{Proof of the results in \hyperref[SECS]{Section \ref{SECS}}} \label{POTR}
\begin{lemma} \label{Main}
For a non negative integer $k$ and $\alpha \in (0,1)$ we have

\begin{equation*}
\sum\limits_{n \leq x}\dfrac{\log^k n}{n^{\alpha}}= I_k(x,\alpha) + H_k(\alpha) + O\left( \dfrac{\log^k x}{x^{\alpha}}  \right)
\end{equation*}
where $$H_k(\alpha)= \int\limits_{1}^{\infty} \{t\}\left( \dfrac{k \log^{k-1} t-\alpha  \log^{k} t}{t^{1+\alpha}}  \right) dt 
+ \dfrac{(-1)^{k+1}k!}{(1-\alpha)^{k+1}}+ \left\lfloor \dfrac{1}{k+1} \right\rfloor$$ and
$$I_k(x, \alpha)=\sum\limits_{i=0}^{k}\left(\dfrac{\Perms{k}{i}(-1)^i\log^{k-i}x}{(1-\alpha)^{i+1}}\right)x^{1-\alpha}. $$
\end{lemma}
\begin{proof}
Apply Abel summation formula \hyperlink{cite.Ap}{\cite[Theorem 4.2]{Ap}} and use induction on $k$.
\end{proof}

\begin{proposition} \label{MLET}
The limit defined as $\zeta_k(\alpha,r,q)$ exists for $\alpha \in (0,1)$.
\end{proposition}
\begin{proof}
Consider the partial sum
$$A(x)=\sum\limits_{\substack{n \leq x \\ n \equiv r \bmod q}}1.$$ Then we have
 \begin{align*}
 H_{k}(x,\alpha,r,q) &= \sum\limits_{\substack{n\leq x \\ n \equiv r\bmod q}} \dfrac{\log^{k} n}{n^{\alpha}}=\int_{1^{-}}^{x}\dfrac{\log^kt}{t^\alpha}dA(t) \\
 &=\int_{1^{-}}^{x}\dfrac{\log^kt}{t^\alpha}d\left(\left(\dfrac{t-r}{q}\right) - \Big\{ \dfrac{t-r}{q} \Big\} \right) \\
 &=\dfrac{1}{q}\int_{1}^{x}\dfrac{\log^kt}{t^\alpha}dt + C + O\left( \dfrac{\log ^k x}{x^\alpha} \right)
 \end{align*}
 for some constant C. Now the result follows from \hyperref[Main]{Lemma \ref{Main}}
\end{proof}
\begin{corollary} \label{MFC}
$\zeta_k(\alpha,r,q)$ satisfy the following properties :
\begin{enumerate}
\item $\zeta_0(\alpha,0,1)= H_0(\alpha)=\zeta(\alpha)$
\item $\zeta_k(\alpha,r \pm q,q)=\zeta_k(\alpha,r,q)$
\item $\sum\limits_{r=1}^{q}\zeta_k(\alpha,r,q)=H_k(\alpha)$
\item \label{4th} $\sum\limits_{j=0}^{q-1}\zeta_k(\alpha,r+jm,mq)=\zeta_k(\alpha,r,m).$
\end{enumerate}
\end{corollary}
\begin{proof}
Part 1, 2 and 3 follow immediately from the definition of $\zeta_k(\alpha,r,q).$ \\
For the part 4 notice that
\begin{align*}
\sum\limits_{j=0}^{q-1}\zeta_k(\alpha,r+jm,mq)&=\sum\limits_{j=0}^{q-1}\lim_{x \rightarrow \infty}\left( \sum\limits_{\substack{0<n\leq x \\ n \equiv r+jm\bmod mq}} \dfrac{\log^k n}{n^{\alpha}} - \dfrac{I_k(x,\alpha)}{mq}\right).
\end{align*}
Here in the right hand side $n=(r+jm)+n'mq$ where $n' \geq 0$ and $j \in \{0,1,2,\cdots,q-1\}$. This means $n=r+(n'q+j)m$ where $n'q+j$ varies over all non-negative integers. Hence 
$$\sum\limits_{j=0}^{q-1}\zeta_k(\alpha,r+jm,mq)=\lim_{x \rightarrow \infty}\left( \sum\limits_{\substack{0<n\leq x \\ n \equiv r\bmod m}} \dfrac{\log ^k n}{n^{\alpha}} - \dfrac{I_k(x,\alpha)}{m}\right)=\zeta_k(\alpha,r,m).$$ 
\end{proof}
\subsection{Proof of \hyperref[MST]{Theorem \ref{MST}}}
Consider the discrete Fourier transform $f$ defined by 
\begin{equation*}
f(a):= \sum\limits_{b=0}^{q-1}\zeta_k(\alpha,b,q)e^{\frac{2 \pi i ab}{q}}.
\end{equation*}
Now for $a=0$ we have $$f(0)=\sum\limits_{b=0}^{q-1}\zeta_k(\alpha,b,q)=H_k(\alpha).$$ On the other hand 
for $a \neq 0$, we have
\begin{align*}
f(a)&=\sum\limits_{b=0}^{q-1}\lim_{x \rightarrow \infty}\left( H_{k}(x,\alpha,b,q) - \dfrac{I_k(x,\alpha)}{q}\right)e^{\frac{2 \pi i ab}{q}} \\
&=\lim_{x \rightarrow \infty}\left(\sum\limits_{b=0}^{q-1}H_{k}(x,\alpha,b,q)e^{\frac{2 \pi i ab}{q}} - \dfrac{I_k(x,\alpha)}{q}\sum\limits_{b=0}^{q-1}e^{\frac{2 \pi i ab}{q}} \right)\\
    &= \lim_{x \rightarrow \infty}\left(\sum\limits_{b=0}^{q-1}\left(\sum\limits_{\substack{0<n\leq x \\ n \equiv b\bmod q}} \dfrac{\ e^{\frac{2 \pi i an}{q}}\log ^k n}{n^{\alpha}}  \right)\right) \\
    &=\sum\limits_{n=1}^{\infty} \dfrac{e^{\frac{2 \pi i an}{q}}\log ^k n}{n^{\alpha}}.  
\end{align*}
Hence we get 
\begin{equation*}
F^{(k)}\left(\frac{a}{q},\alpha\right)=(-1)^k\sum\limits_{b=0}^{q-1}\zeta_k(\alpha,b,q)e^{\frac{2 \pi i ab}{q}}.
\end{equation*}
Applying Fourier inversion on the above identity gives the desired result.  \qed
\subsection{Proof of \hyperref[MSC]{Corollary \ref{MSC}}}
We know that      
\begin{align*}
\zeta_0(\alpha,b,q)&= \dfrac{H_0(\alpha)}{q}+\sum\limits_{a=1}^{q-1}F\left(\frac{a}{q},\alpha\right)e^{\frac{-2\pi i ab}{q}} \\
&= \dfrac{\zeta(\alpha)}{q}+\sum\limits_{a=1}^{q-1}F\left(\frac{a}{q},\alpha\right)e^{\frac{-2\pi i ab}{q}}.
\end{align*}
Now we shall use the following identity given by Ramanujan in \cite{RMJ}
\begin{equation*}
1=\sum\limits_{t \geq 0}(s-1)_t(\zeta(s+t)-1).
\end{equation*} 
 From this one can deduce the analytic continuation for $\zeta(s)$ 
\begin{equation*}
\zeta(s)=\dfrac{s}{s-1}-\dfrac{1}{s-1}\sum\limits_{t\geq 1}(s-1)_t(\zeta(s+t)-1).
\end{equation*}
For $s=\alpha$ where $\Re(\alpha)>0$ it follows from Ramanujan's identity that 
\begin{align*}
\zeta_0(\alpha,b,q)&=\dfrac{\alpha}{q(\alpha-1)}+\dfrac{1}{q(1-\alpha)}\sum\limits_{t \geq 1}(\alpha-1)_t (\zeta(\alpha+t)-1)+\dfrac{1}{q}\sum\limits_{a=1}^{q-1}\sum\limits_{n=1}^{\infty}\dfrac{e^{\frac{2\pi i na}{q}}}{n^{\alpha}}e^{\frac{-2\pi i a b}{q}} \\
&=\dfrac{\alpha}{q(\alpha-1)}+\dfrac{1}{q(1-\alpha)}\sum\limits_{t \geq 1}(\alpha-1)_t \left( \sum\limits_{n=2}^{\infty} \dfrac{1}{n^{\alpha+t}} \right)+\dfrac{1}{q}\sum\limits_{a=1}^{q-1}\sum\limits_{n=1}^{\infty}\dfrac{e^{\frac{2\pi i na}{q}}}{n^{\alpha}}e^{\frac{-2\pi i a b}{q}} \\
&=\dfrac{\alpha}{q(\alpha-1)}+\dfrac{1}{q(1-\alpha)}\sum\limits_{t \geq 1}  \sum\limits_{n=2}^{\infty} \dfrac{(\alpha-1)_t}{n^{\alpha+t}} +\dfrac{1}{q}\sum\limits_{a=1}^{q-1}\sum\limits_{n=1}^{\infty}\dfrac{e^{\frac{2\pi i a(n-b)}{q}}}{n^{\alpha}}. \\
\end{align*}
This completes the proof. \qed 

For a general $k$, since we have analytic continuation of the periodic zeta function $F(x,s)$ we can deduce the following translational formula.
\begin{proposition}
For $l,b,q \in \mathbb{N}$ where $b \leq q$ the function $\zeta_k(\alpha,b,q)$ obtained by meromorphic continuation of the periodic zeta 
function $F(x,s)$ satisfies the following recursive equation 
\begin{equation*}
\zeta_{k+l}(\alpha,b,q)=\dfrac{H_{k+l}(\alpha)}{q}+(-1)^l\left( \zeta_{k}^{(l)}(\alpha,b,q)-\dfrac{H_{k}^{l}(\alpha)}{q} \right)
\end{equation*}
\begin{proof}
Take derivative of equation \eqref{DPZ} $l$ times and compare it with \eqref{DPZ}.
\end{proof}
\end{proposition}

\subsection{Proof of \hyperref[TSF]{Theorem \ref{TSF} }} 
Denote the partial sum \begin{equation*}
A(x):=\sum\limits_{\substack{n \leq x \\ n \equiv r \bmod q}}1=\left \lfloor \dfrac{x-r}{q} \right \rfloor.
\end{equation*}
Then for $\Re(s)>1$, by the use of Riemann Stieltjes integral we have
\begin{align*}
\zeta(s;r,q)&=\int\limits_{1}^{\infty}\dfrac{dA(t)}{t^s}=\dfrac{1}{q}\int\limits_{1}^{\infty}\dfrac{1}{t^s}+\int\limits_{1}^{\infty}\dfrac{1}{t^s}d\left(\left \lfloor \dfrac{t-r}{q} \right \rfloor - \left(\dfrac{t-r}{q} \right)\right) \\
&=\dfrac{1}{q(s-1)}+ \int\limits_{1}^{\infty}\dfrac{1}{t^s}d\left(\left \lfloor \dfrac{t-r}{q} \right \rfloor - \left(\dfrac{t-r}{q} \right)\right). \\
\end{align*}
The above sum is an analytic function in the region $\Re(s)>0$ except for a simple pole at $s=1.$ Hence we have the analytic continuation of $\zeta(s;r,q)$ to this region. Now expand both the function $\frac{1}{s-1}$ and $\frac{1}{t^s}$ as a Taylor series around the point $s=\alpha.$ We get
\begin{align*}
\zeta(s;r,q)&=\dfrac{1}{q}\left(\sum\limits_{m=0}^{\infty}\dfrac{(-1)^m}{(\alpha-1)^{m+1}}(s-\alpha)^m\right) \\
& \quad + \int\limits_{1}^{\infty}\left(\dfrac{(-1)^{m}}{m!}\dfrac{\log^m t}{t^\alpha}(s-\alpha)^m \right)d\left(\left \lfloor \dfrac{t-r}{q} \right \rfloor - \left(\dfrac{t-r}{q} \right)\right) \\
&=\dfrac{1}{q}\left(\sum\limits_{m=0}^{\infty}\dfrac{(-1)^m}{(\alpha-1)^{m+1}}(s-\alpha)^m\right) \\
 & \quad + \sum\limits_{m=0}^{\infty}\dfrac{(-1)^{m}}{m!}(s-\alpha)^{m}\lim_{N \rightarrow \infty}\left( \sum\limits_{\substack{n \leq N \\ n \equiv r \bmod q}}\dfrac{\log^m n}{n^\alpha}-\dfrac{1}{q}\int_{1}^{N}\dfrac{\log^m t}{t^{\alpha}}dt \right).
\end{align*}
Using \hyperref[MLET]{Proposition \ref{MLET}} we get
\begin{align*}
\zeta(s;r,q)&=\dfrac{1}{q}\left(\sum\limits_{m=0}^{\infty}\dfrac{(-1)^m}{(\alpha-1)^{m+1}}(s-\alpha)^m\right)\\
& \quad +\sum\limits_{m=0}^{\infty}\dfrac{(-1)^{m}}{m!}(s-\alpha)^{m}\left(\zeta_m(\alpha,r,q)+ \dfrac{(-1)^m m!}{q(1-\alpha)^{m+1}}\right).
\end{align*}
Simplifying the above sum gives the desired result. \qed
\\ \\
\subsection{Proof of \hyperref[ETT]{Theorem \ref{ETT} }}
For $\Re(s)>1$ we have 
\begin{equation} \label{IEFD}
\zeta(s;r_1,q_1)\zeta(s;r_2,q_2)=L(s,d_{*})=\sum\limits_{n=1}^{\infty}\dfrac{d_{*}(n)}{n^s}=\int\limits_{1^{-}}^{\infty}x^{-s}dA(x)
\end{equation}
where $A(x)$ is given by
\begin{equation} \label{Ax}
A(x)=\sideset{}{'}\sum\limits_{n \leq x}d_{*}(n)=\Delta_{*}(x;r_1,q_1,r_2,q_2)+xP_{k-1}(\log x,r_1,q_1,r_2,q_2).
\end{equation}
By the use of Perron's inversion formula \hyperlink{cite.ARZ}{\cite[Appendix A.3]{ARZ}} the main term can be written as
 \begin{equation*}
 xP_{2}(\log x,r_1,q_1,r_2,q_2)=\text{Res}_{s=1}\left(\zeta(s;r_1,q_1)\zeta(s;r_2,q_2)\dfrac{x^s}{s}\right).
 \end{equation*} 
Further with the help of the Laurent series expansion \eqref{LSH} of $\zeta(s;r,q)$ at $s=1$ we get
\begin{equation*} 
P_{2}(\log x,r_1,q_1,r_2,q_2)=\dfrac{\log x}{q_1q_2}+\dfrac{\gamma_0(r_1,q_1)}{q_2}+\dfrac{\gamma_0(r_2,q_2)}{q_1}-\dfrac{1}{q_1q_2}.
\end{equation*}
Using \eqref{Ax} we can write down the integral in \eqref{IEFD} as the sum of $S_1(s)$ and $S_2(s)$ where
\begin{align*}
S_1(s)&=\int\limits_{1^{-}}^{\infty}x^{-s}d(xP_{2}(\log x,r_1,q_1,r_2,q_2)) \\
&=\dfrac{1}{q_1q_2}\left(\dfrac{1}{(s-1)^2}\right)+\left(\dfrac{\gamma_0(r_1,q_1)}{q_2}+\dfrac{\gamma_0(r_2,q_2)}{q_1}\right)\left(\dfrac{1}{s-1}\right).
\end{align*}
and 
\begin{align*}
S_2(s)&=\int\limits_{1^{-}}^{\infty}x^{-s}d\Delta_{*}(x;r_1,q_1,r_2,q_2)=-\Delta_{*}(1)+s\int\limits_{1}^{\infty}\dfrac{\Delta_{*}(x;r_1,q_1,r_2,q_2)}{x^{s+1}}dx.
\end{align*}
It follows from \eqref{BFDE} and the above expressions for $S_1(s)$ and $S_2(s)$ that the Dirichlet L-series $L(s,d_{*})$ can be analytically continued to the half plane $\Re(s)>\frac{131}{416}$ except for the pole of order 2 at $s=1.$ Using the Taylor series expansion of the function $S_2(s)$ about the point $s=1$ one can deduce the Laurent series expansion for $L(s,d_{*})$ at $s=1$ as follows
\begin{align*} 
L(s,d_{*})&=\dfrac{1}{q_1q_2}\left(\dfrac{1}{(s-1)^2}\right)+\left(\dfrac{\gamma_0(r_1,q_1)}{q_2}+\dfrac{\gamma_0(r_2,q_2)}{q_1}\right)\left(\dfrac{1}{s-1}\right) \\
& \quad + \sum\limits_{k=0}^{\infty}\left(\dfrac{(-1)^k}{k!}\int\limits_{1^{-}}^{\infty}\dfrac{\log ^k x}{x}d\Delta_{*}(x,r_1,q_1,r_2,q_2)\right)(s-1)^{k}.
\end{align*}
To obtain \eqref{IEE} simply equate the coefficient of $(s-1)^k$ in the above Laurent series expansion of $L(s,d_{*})$ to the coefficient of $(s-1)^k$ in the expansion of the product $\zeta(s;r_1,q_1)\zeta(s;r_2,q_2)$ by using \eqref{LSH}. 
Similarly to get \eqref{IEE2} expand the functions $S_1(s)$ and $S_2(s)$ around the point $s=\alpha$ to get a Laurent series expansion for $L(s,d_{*})$ at $s=\alpha$ and equate it to the expansion coming from \hyperref[TSF]{Theorem \ref{TSF}}. This completes the proof. \qed
\section{A generalization of Euler-Stieltjes constant} \label{GOEC}
 The Laurent series expansion \eqref{LSEFP} when $f$ is a Dirichlet character $\chi$ modulo $q$ becomes 
\begin{equation*}
L(s,\chi)=\dfrac{\gamma_{-1}(\chi)}{s-1}+\sum\limits_{k=1}^{\infty}\dfrac{(-1)^{k}\gamma_k(\chi)}{k!}(s-1)^k
\end{equation*}
where $\gamma_k(\chi)=\sum\limits_{r=1}^{q}\chi(r)\gamma_{k}(r,q)$\cite{KNO} for $k \geq 0$ and $\gamma_{-1}(\chi)=\frac{\phi(q)}{q}$ if $\chi$ is principal character $\chi_0$ and 0 otherwise. For a non principal character $\chi$, closed form expressions and explicit upper bounds for the constants $\gamma_k(\chi)$ have been studied by many authors \bb[\citenum{DCH},\citenum{GMAX},\citenum{ISHI},\citenum{KAN},\citenum{SUM},\citenum{TMA}]. For the principal character $\chi_0$ modulo $q$ the expression for the constant $\gamma_{0}(\chi_0)$ was given by Redmond in 1982 \hyperlink{cite.RDN1}{\cite[Lemma 4]{RDN1}} (see also \cite{RDN2}). The general expression for $\gamma_k(\chi_0)$ was given by Shirasaka \hyperlink{cite.SHI}{\cite[Proposition 7]{SHI}}. He proved that 
\begin{equation} \label{SHIC}
\gamma_k(\chi_0)=\sum\limits_{\substack{r=1 \\ (r,q)=1}}^{q}\gamma_k(r,q)
\end{equation} 
and 
\begin{equation}\label{SHIC2}
\gamma_k(\chi_0)=\sum\limits_{j=0}^{k}\dbinom{k}{j}\gamma_jN_{k-j}(q)-\dfrac{1}{k+1}N_{k+1}(q)
\end{equation}
where $$N_{k}(q):=q\sum\limits_{d|q}\frac{\mu(d)\log^kd}{d}.$$
  Observing from the definition \eqref{GELC} of $\gamma_k(r,q)$  and \eqref{SHIC} we state an aysmptotic representation of $\gamma_k(\chi_0)$ as follows. 
\begin{proposition}
For $k \geq 0$ and the Dirichlet character $\chi_0$ modulo $q$, the Laurent series expansion at $s=1$ is  given by
\begin{equation*}
L(s,\chi_0)=\dfrac{\phi(q)}{q(s-1)}+\sum\limits_{k=0}^{\infty}\dfrac{(-1)^k \gamma_k(\chi_{0})}{k!}(s-1)^k
\end{equation*}
where
\begin{equation} \label{LSCR}
\gamma_k(\chi_{0})=\lim_{x\rightarrow \infty} \left( \sum\limits_{\substack{n \leq x \\ (n, \mathrm{rad}(q))=1}}\dfrac{\log^k n}{n} - \prod\limits_{p|\mathrm{rad}(q)}\left(1-\dfrac{1}{p}\right)\dfrac{\log^{k+1} x}{k+1}\right)
\end{equation}
and $\mathrm{rad}(n)$ denotes the radical of $n$.
\end{proposition} 
 In the case when $k=0$ and $q=1$, the expression \eqref{LSCR} reduces to the asymptotic representation \eqref{SRG} of $\gamma_k$ which was first discovered by Stieltjes and then later by many other authors \hyperlink{cite.BLI}{\cite[p. 538]{BLI}}. In 2008, Diamond and Ford \cite{DHF} considered the constants $\gamma(\wp)$ associated to a finite set of primes $\wp$ as follows:
 \begin{equation*}
 \gamma(\wp):=\lim_{x \rightarrow \infty}\left( \sum_{\substack{n \leq x \\ (n,\prod\limits_{p \in \wp}p)=1}}\dfrac{1}{n}-\delta_\wp \log x \right)
 \end{equation*}
where $\delta_{\wp}=\prod\limits_{p \in \wp}\left(1-\dfrac{1}{p} \right).$ It is easy to see that $\gamma(\emptyset)=\gamma.$ A simple observation shows that 
\begin{equation*} 
\gamma(\wp)=\lim_{s \rightarrow 1}\left[\zeta(s)\prod\limits_{p \in \wp}\left(1-\dfrac{1}{p^s}\right)-\dfrac{\delta_{\wp}}{s-1}  \right]=\gamma_0(\chi_0)
\end{equation*} 
where $\chi_0$ is the principal Dirichlet character modulo $q=\prod\limits_{p \in \wp}p.$ Hence Proposition 1 in \cite{DHF} and Lemma 2 in \cite{MZ} are a restatement of Lemma 4 in \cite{RDN1}. However, the constants $\gamma(\wp)$ are of special importance when $\wp=\mathcal{P}_r$ where $\mathcal{P}_r$ is the set of first $r$ primes. This is mainly because Diamond and Ford \hyperlink{cite.DHF}{\cite[Corollary 1]{DHF}} proved that the Riemann Hypothesis is true if and only if $``\gamma(\mathcal{P}_r) > e^{- \gamma}$ for all $r \geq 0".$  Here we consider the behaviour of the constants $\gamma_k(\mathcal{P}_r):=\gamma_k(\chi_0)$ where $\chi_0$ is the principal Dirichlet character modulo $q=\prod_{p \in \mathcal{P}_r} p.$ For further discussion we need an expression for $\gamma_{k}(\mathcal{P}_r)$ in terms of generalized von Mangoldt function $\Lambda_{m}:=\sum\limits_{d|n}\mu(d)\log^k\left(\frac{n}{d}\right)$.
\begin{proposition} \label{CFEG}
For a finite set of primes $S$ denote $P_{S}=\prod\limits_{p \in S}p$. Then we have
$$\gamma_{k}(\mathcal{P}_r)=\prod\limits_{p|P_r}\left(1-\dfrac{1}{p} \right)\left(\sum\limits_{i=0}^{k}\left(\sum\limits_{S \subseteq \mathcal{P}_r}\dfrac{L_{k-i}(P_{S})}{\prod\limits_{p|P_S}(p-1)}\right)\dbinom{k}{i}\gamma_{i} -\dfrac{1}{k+1}\sum\limits_{S \subseteq \mathcal{P}_r}\dfrac{L_{k+1}(P_{S})}{\prod\limits_{p|P_S}(p-1)}\right).$$
where
$$L_k:=\sum\limits_{m=0}^{k}\dbinom{k}{m}(-1)^{m}\log^{k-m}\Lambda_{m}.$$ and $P_r:=\prod\limits_{p \in \mathcal{P}_r}p.$
\end{proposition}
\begin{proof}
Firstly observe that
\begin{align*}
\sum\limits_{d|n}\mu(d)\log^k(d)&=\sum\limits_{d|n}\mu(d)\left(\log(n)-\log\left(\frac{n}{d}\right)\right)^k \\
&=\sum\limits_{m=0}^{k}\dbinom{k}{m}(-1)^{m}\log^{k-m}(n)\Lambda_{m}(n). \\
\end{align*}
With the help of Mobius inversion we can express $\frac{N_{k}(P_r)}{q}$ in the following way:
\begin{align*}
 \dfrac{N_{k}(P_r)}{q}&=\sum\limits_{d|P_r}\dfrac{\mu(d)\log^k(d)}{d} =\sum\limits_{d|P_r}\dfrac{\sum\limits_{d_1|d}L_k(d_1)\mu\left(\dfrac{d}{d_1}\right)}{d} =\sum\limits_{d_1d_2|P_r}\dfrac{L_k(d_1)\mu(d_2)}{d_1d_2} \\
 &=\sum\limits_{d_1|P_r}\dfrac{L_k(d_1)}{d_1}\sum\limits_{d_2|\frac{P_r}{d_1}}\dfrac{\mu(d_2)}{d_2}=\sum\limits_{d_1|P_r}\dfrac{L_k(d_1)}{d_1}\sum\limits_{d_2|\frac{P_r}{d_1}}\dfrac{\mu(d_2)}{d_2}\\
&= \sum\limits_{S \subseteq \mathcal{P}_r}\dfrac{L_k(P_{S})}{P_S}\prod\limits_{p|\frac{P_r}{P_S}}\left(1-\dfrac{1}{p} \right)\prod_{p|P_S}\left(1-\dfrac{1}{p} \right)\prod_{p|P_S}\left(1-\dfrac{1}{p} \right)^{-1} \\
&=\sum\limits_{S \subseteq \mathcal{P}_r}\dfrac{L_k(P_{S})}{P_S}\prod\limits_{p|P_r}\left(1-\dfrac{1}{p} \right)\prod_{p|P_S}\left(1-\dfrac{1}{p} \right)^{-1} \\
&=\prod\limits_{p|P_r}\left(1-\dfrac{1}{p} \right)\sum\limits_{S \subseteq \mathcal{P}_r}\dfrac{L_k(P_{S})}{\prod\limits_{p|P_S}(p-1)}.
\end{align*}
Now use the above expression along with \eqref{SHIC2} to complete the proof.
\end{proof}

  For $r=0$ it is already known from a result of Briggs \hyperlink{cite.WB}{\cite[Theorem 1]{WB}} that for $k \geq 0$, the constant $\gamma_k(\mathcal{P}_0)$ changes sign infinitely often. This was further improved by Mitrovi\'c \hyperlink{cite.MDR}{\cite[Theorem 4]{MDR}}. On the other hand, for a general $r \geq 0$ and $k=0$ it was shown  \hyperlink{cite.DHF}{\cite[Theorem 2]{DHF}} that there are infinitely many integers $r$ for which $\gamma_0(\mathcal{P}_{r+1})>\gamma_0(\mathcal{P}_{r})$ and infinitely many integers $r$ for which $\gamma_0(\mathcal{P}_{r+1})<\gamma_0(\mathcal{P}_{r}).$
More generally using \hyperref[CFEG]{Proposition \ref{CFEG}} one can give a relation between the  two consecutive constants $\gamma_k(\mathcal{P}_{r})$ and $\gamma_k(\mathcal{P}_{r+1})$ as follows.
\begin{corollary}
We have
\begin{equation} \label{IIDF}
\gamma_{k}(\mathcal{P}_{r+1})=\gamma_{k}(\mathcal{P}_{r})\left[\left(1-\dfrac{1}{p_{r+1}}\right)\dfrac{G_k(p_{r+1})}{G_k(p_r)} \right]
\end{equation}
where $G_k$ is the function given by
\begin{equation*}
G_k(x):=\sum\limits_{i=0}^{k}\left(\sum\limits_{S \subseteq \mathcal{P}_{\pi(x)}}\dfrac{L_{k-i}(P_{S})}{\prod\limits_{p|P_S}(p-1)}\right)\dbinom{k}{i}\gamma_{i} -\dfrac{1}{k+1}\sum\limits_{S \subseteq \mathcal{P}_{\pi(x)}}\dfrac{L_{k+1}(P_{S})}{\prod\limits_{p|P_S}(p-1)}
\end{equation*}
and $\pi(x)$ is the prime counting function.
\end{corollary}  
\begin{proof}
From \hyperref[CFEG]{Proposition \ref{CFEG}} we have
\begin{equation*}
\begin{aligned}
&\gamma_{k}(\mathcal{P}_{r+1})=\prod\limits_{p|P_{r+1}}\left(1-\dfrac{1}{p} \right)\left(\sum\limits_{i=0}^{k}\left(\sum\limits_{S \subseteq \mathcal{P}_{r+1}}\dfrac{L_{k-i}(P_{S})}{\prod\limits_{p|P_S}(p-1)}\right)\dbinom{k}{i}\gamma_{i} -\dfrac{1}{k+1}\sum\limits_{S \subseteq \mathcal{P}_{r+1}}\dfrac{L_{k+1}(P_{S})}{\prod\limits_{p|P_S}(p-1)}\right) \\
&=\prod\limits_{p|P_{r}}\left(1-\dfrac{1}{p} \right)\left(1-\dfrac{1}{p_{r+1}} \right)\Biggl(\sum\limits_{i=0}^{k}\left[\sum\limits_{S \subseteq \mathcal{P}_{r}}\dfrac{L_{k-i}(P_{S})}{\prod\limits_{p|P_S}(p-1)}+\dfrac{1}{p_{r+1}-1}\sum\limits_{S \subseteq \mathcal{P}_{r}}\dfrac{L_{k-i}(p_{r+1}P_{S})}{\prod\limits_{p|P_S}(p-1)}\right]\dbinom{k}{i}\gamma_{i}\\
&\qquad -\dfrac{1}{k+1}\sum\limits_{S \subseteq \mathcal{P}_{r}}\dfrac{L_{k+1}(P_{S})}{\prod\limits_{p|P_S}(p-1)}-\dfrac{1}{(k+1)(p_{r+1}-1)}\sum\limits_{S \subseteq \mathcal{P}_{r}}\dfrac{L_{k+1}(p_{r+1}P_{S})}{\prod\limits_{p|P_S}(p-1)}\Biggr) \\
&=\gamma_{k}(\mathcal{P}_{r})\left(\left(1-\dfrac{1}{p_{r+1}}\right)\dfrac{G_k(p_{r+1})}{G_k(p_r)} \right).
\end{aligned}
\end{equation*}
\end{proof}  
The case when $k=0$, the function $A(x)$ \cite[Equation 3.1]{DHF} coincides with $G_0(x)$ and the equation 3.2 in \cite{DHF} follows from the identity \eqref{IIDF}. 

As a continuation, we study the behaviour of $\gamma_k(\mathcal{P}_r)$ when $k=1$ and obtain the following result.
\begin{theorem} \label{MSIT} 
$|\gamma_1(\mathcal{P}_{r+1})|>|\gamma_1(\mathcal{P}_{r})|$ for all but finitely many $r.$ 
\end{theorem}
\begin{proof}
We start with the expression
\begin{align*}
G_1(p_r)&=\sum\limits_{S \subseteq \mathcal{P}_{r}}\dfrac{L_{1}(P_{S})}{\prod\limits_{p|P_S}(p-1)}\gamma_{0} + \sum\limits_{S \subseteq \mathcal{P}_{r}}\dfrac{L_{0}(P_{S})}{\prod\limits_{p|P_S}(p-1)}\gamma_{1} -\dfrac{1}{2}\sum\limits_{S \subseteq \mathcal{P}_{r}}\dfrac{L_{2}(P_{S})}{\prod\limits_{p|P_S}(p-1)} \\
&=\gamma_{1}-\sum\limits_{S \subseteq \mathcal{P}_{r}}\dfrac{\Lambda_{1}(P_S)}{\prod\limits_{p|P_S}(p-1)}\gamma_{0}+\sum\limits_{S \subseteq \mathcal{P}_{r}}\dfrac{\Lambda_1(P_S)\log P_S}{\prod\limits_{p|P_S}(p-1)}-\dfrac{1}{2}\sum\limits_{S \subseteq \mathcal{P}_{r}}\dfrac{\Lambda_{2}(P_{S})}{\prod\limits_{p|P_S}(p-1)} \\
&=\gamma_{1}-\sum\limits_{p \in \mathcal{P}_{r}}\dfrac{\log p}{p-1}\gamma_{0}+\sum\limits_{p \in \mathcal{P}_{r}}\dfrac{\log^2 p}{p-1}-\dfrac{1}{2}\sum\limits_{S \subseteq \mathcal{P}_{r}}\dfrac{\Lambda_{2}(P_{S})}{\prod\limits_{p|P_S}(p-1)}.  \\
\end{align*}
The last sum can be further simplified as
\begin{align*}
\sum\limits_{S \subseteq \mathcal{P}_{r}}\dfrac{\Lambda_{2}(P_{S})}{\prod\limits_{p|P_S}(p-1)}&=\sum\limits_{p \in \mathcal{P}_{r}}\dfrac{\Lambda_{2}(p)}{p-1}+\sum\limits_{u=1}^{r-1}\sum\limits_{v=u+1}^{r}\dfrac{\Lambda_{2}(p_u p_v)}{(p_u -1)(p_v-1)} \\
&=\sum\limits_{p \in \mathcal{P}_{r}}\dfrac{\Lambda_{2}(p)}{p-1}+\sum\limits_{u=1}^{r-1}\sum\limits_{v=u+1}^{r}\dfrac{\sum\limits_{d|p_u p_v}\mu(d)\log^2\left(\frac{p_up_v}{d}\right)}{(p_u -1)(p_v-1)} \\
&=-\sum\limits_{p \in \mathcal{P}_{r}}\dfrac{\log^2 p}{p-1}+2\sum\limits_{u=1}^{r-1}\dfrac{\log p_u}{p_u - 1}\sum\limits_{v=u+1}^{r}\dfrac{\log p_v}{p_v - 1}.
\end{align*}
Now by using the definition of the function $A(x)$ \hyperlink{cite.DHF}{\cite[Equation 3.1]{DHF}}, we get
\begin{equation*}
G_1(p_{r+1})-G_1(p_{r})=\dfrac{\log p_{r+1}}{p_{r+1}-1}\left( \dfrac{3}{2}\log p_{r+1}-A(p_r) \right).
\end{equation*}
Now observe that
\begin{align*}
G_1(p_r)&=\gamma_{1}-\sum\limits_{p \in \mathcal{P}_{r}}\dfrac{\log p}{p-1}\gamma_{0}+\dfrac{3}{2}\sum\limits_{p \in \mathcal{P}_{r}}\dfrac{\log^2 p}{p-1}  -\sum\limits_{u=1}^{r-1}\dfrac{\log p_u}{p_u - 1}\left(A(p_{r})-A(p_{u}) \right) \\
&=\gamma_{1}-\sum\limits_{p \in \mathcal{P}_{r}}\dfrac{\log p}{p-1}\gamma_{0}+\dfrac{3}{2}\sum\limits_{p \in \mathcal{P}_{r}}\dfrac{\log^2 p}{p-1}  -A(p_{r})\sum\limits_{u=1}^{r-1}\dfrac{\log p_u}{p_u - 1}+ \sum\limits_{u=1}^{r-1}\dfrac{A(p_u)\log p_u}{p_u - 1} \\
&=\gamma_{1}+\gamma_{0}^{2}+\dfrac{3}{2}\sum\limits_{p \in \mathcal{P}_{r}}\dfrac{\log^2 p}{p-1}+A(p_r)\dfrac{\log p_{r}}{p_r-1}  -A(p_{r})^2+ \sum\limits_{u=1}^{r-1}\dfrac{A(p_u)\log p_u}{p_u - 1} \\
&=\gamma_{1}+\gamma_{0}^{2}+\dfrac{3}{2}\sum\limits_{p \in \mathcal{P}_{r}}\dfrac{\log^2 p}{p-1}+ \sum\limits_{u=1}^{r}\dfrac{A(p_u)\log p_u}{p_u - 1}-A(p_{r})^2. \\
\end{align*}
From the estimate 
\begin{equation*}
\sum\limits_{n \leq x}\dfrac{\Lambda(n)}{n}=\log x + O(1)
\end{equation*}
and equation 3.7 in \cite{DHF} it follows that 
$$A(x)=\log x + O(1).$$ 
Therefore by the use of partial summation we get
\begin{align*}
\sum\limits_{p \leq x}\dfrac{\log^2 p}{p-1}&=(A(x)-\gamma_{0})\log x - \int\limits_{1}^{x}\dfrac{A(t)-\gamma_{0}}{t}dt\\
&=A(x)\log x - \gamma_{0}\log x- \int\limits_{1}^{x}\dfrac{\log t}{t}dt + O\left(\int\limits_{1}^{x}\dfrac{1}{t}dt \right)\\
&=\dfrac{\log^2 x}{2}+O(\log x).
\end{align*}
Similarly one can deduce that
 \begin{equation*}
\sum\limits_{u=1}^{r}\dfrac{A(p_u)\log p_u}{p_u - 1}=\sum\limits_{u=1}^{r}\dfrac{\log^2 p_u}{p_u - 1}+O(\log p_r)$$ and $$A(p_{r})^2=\log^2 p_r + O(\log p_r).
 \end{equation*}
 Using the above estimates we get 
 $$\dfrac{G_1(p_r)}{p_{r+1}}=\dfrac{1}{p_{r+1}}\left(\dfrac{\log^2 p_{r}}{4}+O(\log p_r)\right)$$
 and for $r$ increasing arbitrarily large we have

\begin{align*}
G_1(p_{r+1})-G_1(p_{r})&=\dfrac{\log p_{r+1}}{p_{r+1}-1}\left( \dfrac{3}{2}\log p_{r+1}-A(p_r) \right) \\
&=\dfrac{1}{p_{r+1}-1}\left( \dfrac{3}{2}\log^2 p_{r+1}-\log p_r \log p_{r+1} +O(\log p_{r+1}) \right) \\
&> \dfrac{1}{p_{r+1}}\left(\dfrac{1}{2}\log^2 p_{r+1} + O(\log p_{r+1})\right)\\
& > \dfrac{G_1(p_{r+1})}{p_{r+1}}.
\end{align*}
 Hence we get the desired inequality.   
\end{proof}
Values of $\gamma_k(\mathcal{P}_r)$ for a few initial values of $r$ and $k$ are given in \hyperref[TBL]{Table 1}.

\subsection{A closed form expression for the first and second generalized Stieltjes constant}
We conclude this section with an another important generalization of Euler constant $\gamma$ arising from the study of Laurent series expansion of Hurwitz zeta function $\zeta(s,x)$ around the point $s=1.$ In particular the expansion is given by 
\begin{equation} \label{HWZE}
\zeta(s,a)=\dfrac{1}{s-1}+\sum\limits_{n=0}^{\infty}\dfrac{(-1)^{n}\gamma_{n}(a)}{n!}(s-1)^{n}
\end{equation}
where the constants $\gamma_n(a)$ (known as Generalized Stieltjes constants) was shown by Berndt \hyperlink{cite.BE}{\cite[Theorem 1]{BE}} to have the following asymptotic representation
\begin{equation*}
\gamma_n(a)=\lim\limits_{N\rightarrow \infty}\left( \sum\limits_{k=0}^{N}\dfrac{\log^n (k+a)}{k+a}-\dfrac{\log^{n+1}(N+a)}{n+1} \right).
\end{equation*} 
Here we can see that when $a=1$, \eqref{HWZE} reduces to \eqref{LSC}. For a good account of history and survey of results related to the constants $\gamma_n(a)$ see \cite{BLI}. As mentioned by Blagouchine that these constants are much less studied than the Euler Stieltjes constants $\gamma_n.$ It was conjectured in 2015 by Blagouchine \hyperlink{cite.BLI2}{\cite[p. 103]{BLI2}} that  ``any generalized Stieltjes constant of the form $\gamma_1(r/q)$, where $r$ and $q$ are positive integers such that $r<q$, may be expressed by means of the Euler's constant $\gamma$, the first Stieltjes constant $\gamma_1$, the logarithm of the $\Gamma$ function at rational argument(s) and some relatively simple, perhaps elementary, function." Using a large number of calculations involving Malmsten's integrals, this conjecture was settled by Blagouchine himself in 2015 by proving \hyperlink{cite.BLI}{\cite[Theorem 1]{BLI}} four equivalent closed form expressions \hyperlink{cite.BLI}{\cite[Eq. 37,50,53,55]{BLI}} for the first generalized Stieltjes constant $\gamma_1(r/q).$ In particular equation 50 in \cite{BLI} is the following identity.
\begin{equation} \label{FGSE}
\begin{split}
\gamma_1\left(\dfrac{r}{q}\right)&=\gamma_1-\gamma \log 2q - \dfrac{\pi}{2}\left(\gamma+\log 2\pi q\right)\cot\left(\dfrac{\pi r}{q}\right)+\sum\limits_{l=1}^{q-1}\cos \dfrac{2\pi r l}{q}\cdot\zeta^{''}\left(0,\dfrac{l}{q}\right) \\
& \quad + \pi\sum\limits_{l=1}^{q-1}\sin\dfrac{2 \pi r l}{q}\cdot\log \Gamma\left(\dfrac{l}{q}\right) + (\gamma+ \log 2\pi q)\sum\limits_{l=1}^{q-1}\cos \dfrac{2\pi r l}{q}\cdot \log \sin \left(\dfrac{\pi l}{q}\right) \\
& \quad - \log^2 2 - \log 2\cdot \log \pi q - \dfrac{1}{2} \log^2 q. 
\end{split}
\end{equation}
The identity \eqref{FGSE} was recently proved by Coffey \hyperlink{cite.COFFM}{\cite[Proposition 1]{COFFM}} using functional equation of Hurwitz zeta function. We give here a much simpler proof with the help of an identity of Shirasaka. \\
\subsubsection*{Proof of \eqref{FGSE}}
Using Laurent series expansion \eqref{LSH} of $\zeta(s;r,q)$ Shirasaka proved that 
\hyperlink{cite.SHI}{\cite[Theorem (ii)]{SHI}} the constants $\gamma_k(r,q)$ are related with generalized Stieltjes constants by the following identity:
\begin{equation} \label{FLTS}
\gamma_k(r,q)=\dfrac{1}{q}\sum\limits_{j=0}^{k}\dbinom{k}{j}(\log q)^{k-j}\gamma_j\left(\dfrac{r}{q}\right)-\dfrac{\log^{k+1} q}{q(k+1)}.
\end{equation}
By the use of the work of Deninger \cite{DCH}, Kanemitsu  gave the closed form expressions for $\gamma_0(r,q)$ \hyperlink{cite.KAN}{\cite[Eq. 3.7]{KAN}}, $\gamma_1(r,q)$ \hyperlink{cite.KAN}{\cite[Eq. 3.10]{KAN}} and $\gamma_2(r,q)$ \hyperlink{cite.KAN}{\cite[Eq. 3.11]{KAN}} in terms of Deninger function $R_k(x)$ which satisfies \hyperlink{cite.KAN}{\cite[Eq. 2.7]{KAN}} 
\begin{equation*}
R_k(x)=(-1)^{k+1}\dfrac{\partial^k}{\partial s^k}\zeta(0,x).
\end{equation*}  
These expressions for $\gamma_0(r,q),\gamma_1(r,q),\gamma_2(r,q)$ and the identity \eqref{FLTS} readily gives the desired expression for $\gamma_j\left(\dfrac{r}{q}\right)$ with $j=0,1,2.$ 
   
\section{Concluding remarks} \label{CONS}
Just like, $\tilde{\Delta}_2(x)$ a more general $\tilde{\Delta}_k(x)$ is defined as the summatory function of all those natural number $n$ which can be written as the product of $k$ many factors with each factor of the form $n_i\equiv r_i \bmod q_i$ for $i=1,2,...,k$. Any interested reader can derive results for $\tilde{\Delta_k}(x)$ parallel to \hyperref[ETT]{Theorem \ref{ETT}} using the results of \cite{IKA} and \cite{SIT}. However the expressions become more complicated and tedious to handle. On the other hand compared to the case of $\gamma(r,q)$ the arithmetic nature of $\zeta_k(\alpha,r,q)$ seems to be even harder to investigate because Baker's theory on linear forms of logarithm of algebraic numbers seems to be of no help. Nevertheless we hope to see some progress in the understanding of analytic and arithmetic nature of these constants in future.    
\section{Acknowledgements}
The authors would like to thank the referee for some highly useful suggestions which improved the quality of this paper. The first author thanks NBHM for providing partial support for this work. The second author thanks CSIR for financial support. 
\newpage
\begin{landscape}
\begin{table} \label{TBL}
\caption{Values for some $\gamma_k(\mathcal{P}_r)$}
\begin{tabular}{l|llllllll}
$r \setminus k$ & $0$ & $1$ & $2$ & $3$ & $4$ & $5$ & $6$ & $7$ \\
\hline
$0$ & $0.577216$ & $-0.0728158$ & $-0.00969036$ & $0.00205383$ & $0.00232537$ & $0.000793324$ & $-0.000238769$ & $-0.000527290$ \\
$1$ & $0.635181$ & $-0.116342$ & $-0.0627675$ & $-0.0453815$ & $-0.0357170$ & $-0.0278208$ & $-0.0208683$ & $-0.0150387$ \\
$2$ & $0.606556$ & $-0.209589$ & $-0.151847$ & $-0.0645267$ & $0.127215$ & $0.524643$ & $1.30859$ & $2.81260$ \\
$3$ & $0.592541$ & $-0.276571$ & $-0.215951$ & $0.00931819$ & $0.669435$ & $2.12763$ & $4.17280$ & $1.97223$ \\
$4$ & $0.582022$ & $-0.329654$ & $-0.263695$ & $0.136174$ & $1.41389$ & $3.95177$ & $3.70916$ & $-31.0884$ \\
$5$ & $0.578937$ & $-0.366822$ & $-0.308362$ & $0.244198$ & $2.21012$ & $6.06838$ & $1.19016$ & $-95.6539$ \\
$6$ & $0.575402$ & $-0.400252$ & $-0.345953$ & $0.376805$ & $3.10193$ & $8.11685$ & $-5.99605$ & $-201.331$ \\
$7$ & $0.573521$ & $-0.427320$ & $-0.380680$ & $0.498397$ & $4.00756$ & $10.2843$ & $-15.4953$ & $-339.655$ \\
$8$ & $0.571312$ & $-0.452521$ & $-0.411088$ & $0.634850$ & $4.96629$ & $12.2857$ & $-29.6893$ & $-517.649$ \\
$9$ & $0.569788$ & $-0.474179$ & $-0.439333$ & $0.763224$ & $5.92289$ & $14.3212$ & $-46.0836$ & $-725.432$ \\
$10$ & $0.569135$ & $-0.492008$ & $-0.466438$ & $0.871617$ & $6.84283$ & $16.5601$ & $-62.3404$ & $-953.461$ \\
$11$ & $0.568272$ & $-0.509140$ & $-0.491384$ & $0.988428$ & $7.79444$ & $18.6716$ & $-82.3452$ & $-1215.83$ \\
\end{tabular}
\end{table}
\end{landscape}
\bibliographystyle{amsplain}
\providecommand{\bysame}{\leavevmode\hbox to3em{\hrulefill}\thinspace}
\providecommand{\MR}{\relax\ifhmode\unskip\space\fi MR }
\providecommand{\MRhref}[2]{%
  \href{http://www.ams.org/mathscinet-getitem?mr=#1}{#2}
}
\providecommand{\href}[2]{#2}

\end{document}